\documentclass[12pt]{amsart}
\usepackage{url,amsmath,amssymb,amsthm,mathtools}
\calclayout
\usepackage{tikz}
\usetikzlibrary{arrows.meta, positioning, calc}
\newtheorem{theoremcounter}{Theorem Counter}[section]
\theoremstyle{plain}
\newtheorem{theorem}[theoremcounter]{Theorem}
\newtheorem{lemma}[theoremcounter]{Lemma}
\theoremstyle{remark}
\newtheorem{remark}[theoremcounter]{Remark}
\newcommand{\bn}{\boldsymbol{n}}
\newcommand{\cA}{\mathcal{A}}

\newcommand{\bk}{\boldsymbol{k}}
\newcommand{\bl}{\boldsymbol{l}}
\newcommand{\bm}{\boldsymbol{m}}
\newcommand{\bp}{\boldsymbol{p}}
\newcommand{\PP}{\mathcal{P}}
\newcommand{\RR}{\mathbb{R}}
\newcommand{\ZZ}{\mathbb{Z}}
\newcommand{\dep}{\mathrm{dep}}
\newcommand{\wt}{\mathrm{wt}}
\numberwithin{equation}{section}
\address{Department of Information and Computer Science, Kanazawa Institute of Technology, 7-1, Ohgigaoka, Nonoichi, Ishikawa, 921-8501, Japan}
\email{c1117334@planet.kanazawa-it.ac.jp}

\address{Department of Mathematical Sciences, Aoyama Gakuin University, 5-10-1 Fuchinobe, Chuo-ku, Sagamihara, Kanagawa, 252-5258, Japan}
\email{seki@math.aoyama.ac.jp}

\address{Komaba Toho High School, 4-5-1, Ikejiri, Setagaya-ku, Tokyo, 154-0001, Japan}
\email{taiki11235@gmail.com}

\thanks{This research was supported by JSPS KAKENHI Grant Number 21K13762.}

\title[Deriving two dualities simultaneously]{Deriving two dualities simultaneously from a family of identities for multiple harmonic sums}
\author{Takumi Maesaka, Shin-ichiro Seki, and Taiki Watanabe}
\date{}

\begin{document}

\maketitle

\begin{abstract}
We give a new expression of the multiple harmonic sum, which serves as a refinement of the iterated integral expression of the multiple zeta value, and prove it using the so-called connected sum method.
Based on this fact, by taking two kinds of limit operations, we obtain new proofs of both the duality for multiple zeta values and the duality for finite multiple zeta values.
\end{abstract}

\setcounter{section}{-1}
\section{Introduction}
The following have kept its place as fundamental in the study of multiple zeta values for about thirty years.
\begin{theorem}[Iterated integral expression of MZV; widely attributed to Drinfel'd and Kontsevich, and explicitly written in literature by Le--Murakami~\cite{LeMurakami2005}]\label{thm:integral}
Let $k_1,\dots, k_r$ be positive integers with $k_r>1$.
Set $k\coloneqq k_1+\cdots+k_r$, $\omega_0(t)\coloneqq\frac{\mathrm{d}t}{t}$, and $\omega_1(t)\coloneqq\frac{\mathrm{d}t}{1-t}$.
Then, we have
\[
\zeta(k_1,\dots,k_r)=\int_{0<t_1<\cdots<t_k<1}\omega_{a_1}(t_1)\cdots\omega_{a_k}(t_k),
\]
where, for $i\in\{1, k_1+1, k_1+k_2+1, \dots, k_1+\cdots+k_{r-1}+1\}$, $a_i=1$ and for all other $i$, $a_i=0$.
\end{theorem}
This integral is improper, but can be approximated by certain Riemann sums:
\[
\int_{0<t_1<\cdots<t_k<1}\omega_{a_1}(t_1)\cdots\omega_{a_k}(t_k)=\lim_{N\to\infty}\sum_{0<n_1<\cdots<n_k<N}\hat{\omega}^{(N)}_{a_1}(n_1)\cdots\hat{\omega}^{(N)}_{a_k}(n_k),
\]
where $\hat{\omega}^{(N)}_0(n)\coloneqq\frac{1/N}{n/N}=\frac{1}{n}$ and $\hat{\omega}^{(N)}_1(n)\coloneqq\frac{1/N}{1-n/N}=\frac{1}{N-n}$.
If one tries truncating both sides of the expression in Theorem~\ref{thm:integral}, the multiple harmonic sum and the Riemann sum do not coincide in general:
\[
\zeta^{}_{<N}(k_1,\dots,k_r)\neq\sum_{0<n_1<\cdots<n_k<N}\hat{\omega}^{(N)}_{a_1}(n_1)\cdots\hat{\omega}^{(N)}_{a_k}(n_k).
\]
This non-coincidence is not surprising, because there is no necessity for the values of two convergent sequences to coincide when they converge to the same limit.
However, surprisingly, we found that by appropriately changing the range of the summation of the Riemann sum to a range that is nontrivial yet systematic, the equality can be achieved for \emph{all} $N$. (See Theorem~\ref{thm:main}).
Unfortunately, this paper does not provide an adequate explanation for why such a phenomenon occurs, but we give a proof of the equality by finding an appropriate ``connector.''
We hope that our main result will have numerous applications; however, in this paper, we present a few applications which involve deriving the duality of multiple zeta values and the duality of finite multiple zeta values\footnote{Another application has already been obtained for the extended double shuffle relation \cite{Seki2024}.}.
It is intriguing that there exists a hidden connection between these two dualities, not merely analogous but closely related that they can be derived from a single formula.
\section{Main result}
Let $N$ be a positive integer.
We call a tuple of integers $\bk=(k_1,\dots,k_r)$ an index.
We call $\wt(\bk)\coloneqq k_1+\cdots+k_r$ its weight and $\dep(\bk)\coloneqq r$ its depth.
Then the \emph{multiple harmonic sum} $\zeta_{<N}^{}(\bk)$ is defined by
\[
\zeta_{<N}^{}(\bk)\coloneqq\sum_{0<n_1<\cdots<n_r<N}\frac{1}{n_1^{k_1}\cdots n_r^{k_r}}.
\]
Based on the multiple harmonic sum, the \emph{multiple zeta value} and the \emph{finite multiple zeta value} are defined through two different operations.
When $k_r\geq 2$, the index $\bk$ is called an admissible index.
For an admissible index $\bk$, the multiple zeta value $\zeta(\bk)$ is defined as
\[
\zeta(\bk)\coloneqq\lim_{N\to\infty}\zeta_{<N}^{}(\bk)\in\RR.
\]
For an index $\bk$, not necessarily admissible, the finite multiple zeta value $\zeta_{\cA}^{}(\bk)$ is defined as
\[
\zeta_{\cA}^{}(\bk)\coloneqq(\zeta_{<p}^{}(\bk)\bmod{p})_{p\in\PP}\in\cA,
\]
where $\cA$ is defined as
\[
\cA\coloneqq\left.\left(\prod_{p\in\PP} \ZZ/p\ZZ\right) \middle/ \left(\bigoplus_{p\in\PP} \ZZ/p\ZZ\right) \right.
\]
and $\PP$ is the set of all prime numbers.

In recent years, there are numerous studies on relations among multiple zeta values in $\RR$ and relations among finite multiple zeta values in $\cA$.
There are families of relations that, due to their similar forms, are referred to by the same name such as the ``sum formula'' in each world.
See \cite{Kaneko2019, Zhao2016} for various known relations.
The main subject of this paper is the families of relations known as the ``duality.''
Let us recall the duality in each world.

Each admissible index $\bk$ is uniquely expressed as
\[
\bk=(\{1\}^{a_1-1},b_1+1,\dots,\{1\}^{a_s-1},b_s+1),
\]
where $s$ and $a_1$, $\dots$, $a_s$, $b_1$, $\dots$, $b_s$ are positive integers.
Here, $\{1\}^a$ means $1,\dots, 1$ repeated $a$ times.
Then the dual index $\bk^{\dagger}$ of $\bk$ is defined as
\[
\bk^{\dagger}\coloneqq(\{1\}^{b_s-1},a_s+1,\dots,\{1\}^{b_1-1},a_1+1).
\]
\begin{theorem}[Duality in $\RR$]\label{thm:Rduality}
For an admissible index $\bk$, we have
\[
\zeta(\bk)=\zeta(\bk^{\dagger}).
\]
\end{theorem}
This was conjectured by Hoffman~\cite{Hoffman1992} and was proved using Theorem~\ref{thm:integral} as explained in \cite{Zagier1994}.
Several proofs of this duality are known, such as \cite{Kawashima2009} (deducing from Kawashima's relation), \cite{BorweinChan2010} (induction), \cite{HiroseSato2019} (deducing from the confluence relations), and \cite{SekiYamamoto2019} (connected sum method).

For an index
\[
\bk=(\{1\}^{a_1-1},b_1+1,\dots,\{1\}^{a_{s-1}-1},b_{s-1}+1,\{1\}^{a_s-1},b_s),
\]
the Hoffman dual $\bk^{\vee}$ of $\bk$ is defined as
\[
\bk^{\vee}\coloneqq(a_1,\{1\}^{b_1-1},a_2+1,\{1\}^{b_2-1},\dots,a_s+1,\{1\}^{b_s-1}),
\]
where $s$ and $a_1$, $\dots$, $a_s$, $b_1$, $\dots$, $b_s$ are positive integers.

For two indices $\bk$ and $\bl$, the relation $\bl\preceq\bk$ means $\bl$ obtained by replacing some commas in $\bk=(k_1,\dots,k_r)$ by plus signs.
For example, $(5,1)\preceq(2,3,1)$, $(9)\preceq(1,4,2,2)$, and $\bk\preceq\bk$.
Then we define the \emph{finite multiple zeta star value} $\zeta_{\cA}^{\star}(\bk)$ as
\[
\zeta_{\cA}^{\star}(\bk)\coloneqq\sum_{\bl\preceq\bk}\zeta_{\cA}^{}(\bl).
\]
\begin{theorem}[Duality in $\cA$]\label{thm:Aduality}
For an index $\bk$, we have
\[
\zeta_{\cA}^{\star}(\bk)=-\zeta_{\cA}^{\star}(\bk^{\vee}).
\]
\end{theorem}
This relation is also known as Hoffman's duality because this was first proved by Hoffman \cite[Theorem~4.6]{Hoffman2015}.
Several proofs of this duality are known, such as \cite{Imatomi2014} (Landen type connection), \cite{Yamamoto2017} (Yamamoto's integral expression), and \cite{SekiYamamoto2020} (connected sum method).

Although the above two dualities are called by the same name, it is not clear to what extent they are mathematically similar.
For instance, there are known families of relations called the ``cyclic sum formula'' in $\RR$ (\cite{HoffmanOhno2003}) and in $\cA$ (\cite{KawasakiOyama2020}) with the same name, and as explained in \cite[\S7]{Seki2020}, the fact that both can be derived simultaneously from the same connected sum lends credibility of their way of naming.
As another example, the duality satisfied by the so-called symmetric multiple zeta values (see \cite[Definition~9.3]{Kaneko2019} for the definition) and the duality in $\cA$ have been derived simultaneously from a single formula for finite multiple harmonic $q$-series in \cite{BachmannTakeyamaTasaka2018}.
Indeed, the duality for symmetric multiple zeta values has the same form as the duality in $\cA$.
On the other hand, the forms of the dualities in $\RR$ and $\cA$ are different, and additionally, the proofs for these dualities using the connected sum method explained in \cite[\S5, \S6]{Seki2020} employ different connected sums, giving rise to the sense that, although they share the same name, they may in fact be different phenomena.

In this context, several years ago, Minoru Hirose asked the second author the following question in a private communication: ``Is it possible to derive the duality in $\RR$ and the duality in $\cA$ simultaneously?''
This paper gives an \emph{affirmative} answer to this question.

In order to achieve this, we state the new formula mentioned in Introduction.
First, let us prepare some notation.
For a positive integer $n$, $[n]\coloneqq\{1,2,\dots,n\}$.
For a positive integer $k$ and an index $\bn=(n_1,\dots,n_k)$ with $n_1<N$, we define $P_k^{(N)}(\bn)\in\mathbb{Q}$ as
\[
P_k^{(N)}(\bn)\coloneqq\frac{1}{(N-n_1)n_2\cdots n_k}.
\]
When $k=1$, $P_1^{(N)}(n_1)$ is not $1/n_1$ but $1/(N-n_1)$.
For an index $\bk=(k_1,\dots, k_r)$ with $\wt(\bk)=k$, we define sets $J(\bk)$ and $S_N(\bk)$ as
\[
J(\bk)\coloneqq\{1,k_1+1,k_1+k_2+1,\dots,k_1+\cdots+k_{r-1}+1, k+1\}
\]
and
\[
S_N(\bk)\coloneqq\left\{(n_1,\dots,n_k)\in\ZZ^k \ \middle| \ \begin{array}{cc} n_{i-1}<n_i  & \text{ if } i\in J(\bk), \\ n_{i-1}\leq n_i & \text{ if } i\in [k]\setminus J(\bk), \\ \text{where }n_0=0, & n_{k+1}=N \end{array}\right\},
\]
respectively.
When indices are joined by commas, it means the concatenation, and it is understood that parentheses are appropriately removed.
For example, $((2,3),(3,1))=(2,3,3,1)$ and $((3,5),1)=(3,5,1)$.
Finally, we define a kind of multiple harmonic sum in the form of a Riemann sum, denoted as $\zeta_{<N}^{\flat}(\bk)$, as follows:
\[
\zeta_{<N}^{\flat}(\bk)\coloneqq\sum_{\substack{(\bn_1,\dots,\bn_r)\in S_N(\bk) \\ \dep(\bn_i)=k_i \text{ for all }i\in[r]}}\prod_{i=1}^rP_{k_i}^{(N)}(\bn_i).
\]
For example,
\[
\zeta_{<N}^{\flat}(2,3)=\sum_{0<n_1\leq n_2<n_3\leq n_4\leq n_5<N}\frac 1{(N-n_1)n_2(N-n_3)n_4n_5}.
\]
We prove that two kinds of multiple harmonic sums coincide.
\begin{theorem}[Discretization of the iterated integral expression of MZV]\label{thm:main}
For a positive integer $N$ and an index $\bk$, we have
\[
\zeta_{<N}^{}(\bk)=\zeta^{\flat}_{<N}(\bk).
\]
\end{theorem}
When $\bk$ consists of a sequence of $1$s, that is, for some $a$, $\bk=(\{1\}^a)$, this appears in Granville's paper (in his notation as $B(m,g-j)$).
Based on this family of identities, it is possible to obtain new proofs of the duality in $\RR$ and the duality in $\cA$.

\

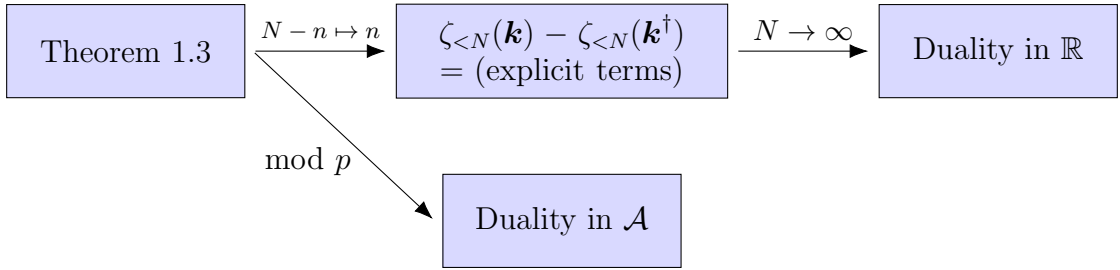
\begin{center}
\begin{tikzpicture}[node distance=1cm and 2cm, auto]
   \tikzset{
       block/.style={rectangle, draw, text width=7em, text centered, minimum height=3em, fill=blue!15},
       line/.style={draw, -{Latex[length=2.5mm]}, shorten >=1pt, shorten <=1pt},
       wide block/.style={block, text width=10em, align=center, fill=blue!15} 
   }

   \node [block] (theorem) {Theorem~\ref{thm:main}};
   \node [wide block, right=of theorem] (zpz) {\( \zeta_{<N}(\bk)-\zeta_{<N}(\bk^{\dagger})\)\\ = (explicit terms)};
   \node [block, below=of zpz] (dualityA) {Duality in $\cA$};
   \node [block, right=of zpz] (dualityR) {Duality in $\RR$};

   \path [line] ($(theorem.east)+(0.1,0)$) -- node[above, midway, font=\scriptsize] {$N-n\mapsto n$} ($(zpz.west)-(0.1,0)$);
   \path [line] ($(zpz.east)+(0.1,0)$) -- node[above, midway, font=\small] {$N\to \infty$} ($(dualityR.west)-(0.1,0)$);
   \draw [line] ($(theorem.east)+(0.1,0)$) -- node[below, midway] {$\mod{p}\quad\quad\quad$} ($(dualityA.west)-(0.1,0)$);

\end{tikzpicture}
\end{center}

\

Therefore, this formula knows both dualities and this could be considered as providing an answer to Hirose's question.

In Section~\ref{sec:Rderiving}, we will derive the duality in $\RR$ from Theorem~\ref{thm:main} and in Section~\ref{sec:Aderiving}, we will derive the duality in $\cA$ from Theorem~\ref{thm:main}.
After checking these, in Section~\ref{sec:connector_proof}, we will prove Theorem~\ref{thm:main}.
In Section~\ref{sec:lifting}, we consider $\bp$-adic liftings of the duality in $\cA$.
\begin{remark}
In the first draft of this paper, we mentioned the similarity between the identity in Theorem~\ref{thm:main} and Kawashima's identity \cite[Proposition~3.2]{Kawashima.pre}, stating that revealing the relationship between them as future research.
However, after reading our first draft, Yamamoto has already provided an answer \cite{Yamamoto2024}.
\end{remark}
As a simple example of discretization other than MZVs, the following holds:
\[
\log2=\sum_{n=1}^{\infty}\frac{(-1)^{n-1}}{n}=\int_0^1\frac{\mathrm{d}t}{1+t} \quad \xrightarrow{\text{discretization}} \quad \sum_{n=1}^{2N-1}\frac{(-1)^{n-1}}{n}=\sum_{n=0}^{N-1}\frac{1}{N+n}.
\]
Is there any law of discretization that holds for a wider subclass of periods?
Here, ``periods'' refers to those in the sense of Kontsevich--Zagier \cite{KontsevichZagier2001}.
\section{Deriving the duality in $\RR$}\label{sec:Rderiving}
We prepare one lemma for use in error estimation.
\begin{lemma}\label{lem:limit}
Let $k$ be a positive integer and $a_1, b_1, \dots, a_k, b_k$ non-negative integers satisfying $a_1,b_k\geq 1$ and $a_i+b_i\geq 1$ for all $i\in[k]$.
If for at least one $i\in[k]$, $a_i+b_i\geq 2$ holds, then we have
\[
\lim_{N\to\infty}\sum_{0<n_1<n_2<\cdots<n_k<N}\prod_{i=1}^k\frac{1}{(N-n_i)^{a_i}n_i^{b_i}}=0.
\]
\end{lemma}
\begin{proof}
Let $c_1, d_1, \dots, c_k, d_k$ be $0$ or $1$, where for each $i\in[k]$, we can choose $c_i$ and $d_i$ such that $c_i+d_i=1$, making possible the subsequent argument.
By applying the transformation $N-n_i\mapsto n_{k+1-i}$ as necessary, we can get the following estimate:
\[
\sum_{0<n_1<n_2<\cdots<n_k<N}\prod_{i=1}^k\frac{1}{(N-n_i)^{a_i}n_i^{b_i}}\leq\sum_{0<n_1<n_2<\cdots<n_k<N}\frac{1}{(N-n_1)n_1}\prod_{i=2}^k\frac{1}{(N-n_i)^{c_i}n_i^{d_i}}.
\]
By using a partial fraction decomposition
\[
\frac{1}{(N-n_1)n_1}=\frac{1}{N}\left(\frac{1}{n_1}+\frac{1}{N-n_1}\right)
\]
and an estimate
\[
\sum_{0<n_1<n_2<\cdots<n_k<N}\prod_{i=1}^k\frac{1}{(N-n_i)^{c_i}n_i^{d_i}}\leq \sum_{0<n_1,\dots,n_k<N}\frac{1}{n_1\cdots n_k}=O(\log^kN),
\]
we can conclude the statement.
\end{proof}
Before giving the general proof, let us look at the case of $\bk=(3)$ as a special case.
By Theorem~\ref{thm:main}, we have
\[
\zeta_{<N}^{}(3)=\zeta_{<N}^{\flat}(3)=\sum_{0<n_1\leq n_2\leq n_3<N}\frac{1}{(N-n_1)n_2n_3}
\]
and
\begin{align*}
\zeta_{<N}^{}(1,2)=\zeta_{<N}^{\flat}(1,2)&=\sum_{0<m_1 < m_2\leq m_3<N}\frac{1}{(N-m_1)(N-m_2)m_3}\\
&=\sum_{0<n_1\leq n_2 < n_3<N}\frac{1}{(N-n_1)n_2n_3}.
\end{align*}
Here we have set $n_1=N-m_3$, $n_2=N-m_2$, and $n_3=N-m_1$.
Hence we can compute
\begin{align*}
\zeta_{<N}^{}(3)-\zeta_{<N}^{}(1,2)&=\sum_{0<n_1<n_3<N}\frac{1}{(N-n_1)n_1n_3}+\sum_{0<n_1<n_2<N}\frac{1}{(N-n_1)n_2^2}\\
&\quad+\sum_{0<n_1<N}\frac{1}{(N-n_1)n_1^2}-\sum_{0<n_1<n_3<N}\frac{1}{(N-n_1)n_1n_3}.
\end{align*}
Since we can apply Lemma~\ref{lem:limit} for each sum in the right-hand side, we get $\zeta(3)=\zeta(1,2)$.
In this case, we also obtain an expression
\[
\zeta_{<N}^{}(3)-\zeta_{<N}^{}(1,2)=\sum_{0<n_1\leq n_2<N}\frac{1}{(N-n_1)n_2^2}
\]
and this coincides with \cite[Equation~(3.1)]{BorweinBradley2006}.
Therefore, the proof of the duality provided below can be considered as a generalization of the second proof of $\zeta(3)=\zeta(1,2)$ in Borwein--Bradley's beautiful article \cite{BorweinBradley2006}.

For positive integers $k$ and $N$, we set
\[
T_N(k)\coloneqq\{(n_1,\dots,n_k)\in\ZZ^k \mid 0<n_1<\cdots<n_k<N\}.
\]
\begin{proof}[Proof of Theorem~$\ref{thm:Rduality}$ using Theorem~$\ref{thm:main}$]
Let $\bk=(k_1,\dots,k_r)$ be an admissible index with $\wt(\bk)=k$.
By Theorem~\ref{thm:main}, we have
\[
\zeta_{<N}^{}(\bk)=\zeta_{<N}^{\flat}(\bk)=\sum_{\substack{(\bn_1,\dots,\bn_r)\in S_N(\bk) \\ \dep(\bn_i)=k_i \text{ for all }i\in[r]}}\prod_{i=1}^rP_{k_i}^{(N)}(\bn_i).
\]
After applying the same procedure to its dual index $\bk^{\dagger}$, by applying the transformation $N-n_i\mapsto n_{k+1-i}$ for all $i\in[k]$, we obtain
\[
\zeta_{<N}^{}(\bk^{\dagger})=\zeta_{<N}^{\flat}(\bk^{\dagger})=\sum_{\substack{(\bn_1,\dots,\bn_r)\in S'_N(\bk) \\ \dep(\bn_i)=k_i \text{ for all }i\in[r]}}\prod_{i=1}^rP_{k_i}^{(N)}(\bn_i)
\]
for some $S'_N(\bk)\subset \ZZ^k$.
Both $S_N(\bk)$ and $S'_N(\bk)$ include $T_N(k)$ as a subset and we can check that the right-hand side of
\begin{equation}\label{eq:discrepancy}
\zeta_{<N}^{}(\bk)-\zeta_{<N}^{}(\bk^{\dagger})=\left(\sum_{\substack{(\bn_1,\dots,\bn_r)\in S_N(\bk)\setminus T_N(k) \\ \dep(\bn_i)=k_i \text{ for all }i\in[r]}}-\sum_{\substack{(\bn_1,\dots,\bn_r)\in S'_N(\bk)\setminus T_N(k) \\ \dep(\bn_i)=k_i \text{ for all }i\in[r]}}\right)\prod_{i=1}^rP_{k_i}^{(N)}(\bn_i)
\end{equation}
is expressed as a finite sum of plus or minus one times sums of the form appearing in Lemma~\ref{lem:limit}.
Therefore, by Lemma~\ref{lem:limit}, we have
\[
\lim_{N\to\infty}\bigl(\zeta_{<N}^{}(\bk)-\zeta_{<N}^{}(\bk^{\dagger})\bigr)=0
\]
and this is the duality in $\RR$.
\end{proof}
This not only provides a new proof of the duality in $\RR$ by series manipulation but also, \eqref{eq:discrepancy} offers a method to concretely describe the discrepancy between $\zeta_{<N}^{}(\bk)$ and $\zeta_{<N}^{}(\bk^{\dagger})$.
\section{Deriving the duality in $\cA$}\label{sec:Aderiving}
It is known that Theorem~\ref{thm:Aduality} is equivalent to the following relations.
\begin{theorem}[{Hoffman~\cite[Theorem~4.7]{Hoffman2015}}]\label{thm:nonstarduality}
For an index $\bk$, we have
\[
\zeta_{\cA}^{}(\bk)=(-1)^{\dep(\bk)}\sum_{\bk\preceq\bl}\zeta_{\cA}^{}(\bl).
\]
\end{theorem}
We prove Theorem~\ref{thm:nonstarduality} directly from Theorem~\ref{thm:main}.
The proof essentially relies only on
\begin{equation}\label{eq:easiest}
\frac{1}{p-n}\equiv-\frac{1}{n}\pmod{p}
\end{equation}
for a prime number $p$ and a positive integer $n<p$.
\begin{proof}[Proof of Theorem~$\ref{thm:nonstarduality}$ using Theorem~$\ref{thm:main}$]
Let $p$ be a prime number.
By definition, for $\bk=(k_1,\dots, k_r)$ with $\wt(\bk)=k$, we compute
\[
\zeta_{<p}^{\flat}(\bk)=\sum_{\substack{(\bn_1,\dots,\bn_r)\in S_p(\bk) \\ \dep(\bn_i)=k_i \text{ for all }i\in[r]}}\prod_{i=1}^rP_{k_i}^{(p)}(\bn_i)\equiv(-1)^r\sum_{(m_1,\dots,m_k)\in S_p(\bk)}\frac{1}{m_1\cdots m_k}\pmod{p}.
\]
Since
\[
\sum_{(m_1,\dots,m_k)\in S_p(\bk)}\frac{1}{m_1\cdots m_k}=\sum_{\bk\preceq\bl}\zeta_{<p}^{}(\bl)
\]
holds, the proof completes.
\end{proof}
Hoffman's original proof of Theorem~\ref{thm:Aduality} was based on an identity (\cite[Theorem~4.2]{Hoffman2015})
\begin{equation}\label{eq:Hoffman_identity}
\sum_{1\leq n_1\leq \cdots \leq n_r\leq N}\frac{1}{n_1^{k_1}\cdots n_r^{k_r}}=\sum_{1\leq m_1\leq \cdots\leq m_s\leq N}(-1)^{m_s-1}\binom{N}{m_s}\frac{1}{m_1^{l_1}\cdots m_s^{l_s}}
\end{equation}
and a congruence
\[
(-1)^{m-1}\binom{p-1}{m}\equiv -1\pmod{p},
\]
where $(k_1,\dots,k_r)^{\vee}=(l_1,\dots,l_s)$, $p$ is a prime number, and $m$ is a positive integer less than $p$.
In the present study, obtaining a proof of Hoffman's duality from an identity without binomial coefficients is, at least for the authors, an unexpected phenomenon.
\section{Proof of Theorem~\ref{thm:main}}\label{sec:connector_proof}
We prove Theorem~\ref{thm:main} by using the connected sum method, which was introduced in \cite{SekiYamamoto2019}.
See \cite{Seki2020} for more details.
The flow of the proof is as follows: Find an appropriate ``connector'' and define the ``connected sum'' to connect $\zeta_{<N}^{}(\bk)$ and $\zeta_{<N}^{\flat}(\bk)$.
Then, components of the index within $\zeta_{<N}^{}(\bk)$ are progressively transported through this connected sum towards $\zeta_{<N}^{\flat}(\bk)$, ultimately resulting in $\zeta_{<N}^{}(\bk)$ completely transitioning into $\zeta_{<N}^{\flat}(\bk)$.

Fix a positive integer $N$.
For non-negative integers $n$ and $m$ satisfying $n\leq m\leq N$, we define a \emph{connector} $C_N(n,m)$ as
\[
C_N(n,m)\coloneqq\frac{\binom{m}n}{\binom{N}n}.
\]
For two indices $\bk=(k_1,\dots,k_r)$ and $\bn=(n_1,\dots,n_r)$, we define $\mathrm{st}(\bn)$, $\mathrm{end}(\bn)$, and $Q_{\bk}(\bn)$ as
\[
\mathrm{st}(\bn)\coloneqq n_1,\quad \mathrm{end}(\bn)\coloneqq n_r,\quad Q_{\bk}(\bn)\coloneqq\frac{1}{n_1^{k_1}\cdots n_r^{k_r}},
\]
respectively.
Let $\bk=(k_1,\dots,k_r)$ and $\bl=(l_1,\dots,l_s)$ be indices with $\wt(\bl)=l$.
We set $\widetilde{J}(\bl)$ and $\widetilde{S}_N(\bk;\bl)$ by
\[
\widetilde{J}(\bl)\coloneqq\{1,l_1+1,l_1+l_2+1,\dots,l_1+\cdots+l_{s-1}+1\}
\]
and
\[
\widetilde{S}_N(\bk;\bl)\coloneqq\left\{(n_1,\dots,n_r,m_1,\dots,m_l)\in\ZZ^{r+l} \ \middle| \ \begin{array}{l} 0<n_1<\cdots<n_r\leq m_1, \\ m_j<m_{j+1}  \quad \text{ if } j\in \widetilde{J}(\bl), \\ m_{j}\leq m_{j+1} \quad \text{ if } j\in [l]\setminus \widetilde{J}(\bl), \\ \text{where }m_{l+1}=N \end{array}\right\},
\]
respectively.
Then we define a \emph{connected sum} $Z_N(\bk\mid\bl)$ as
\[
Z_N(\bk \mid \bl)\coloneqq\sum_{\substack{(\bn,\bm_1,\dots,\bm_s)\in\widetilde{S}_N(\bk;\bl) \\ \dep(\bn)=r, \ \dep(\bm_j)=l_j \text{ for all } j\in[s]}}Q_{\bk}(\bn)\cdot C_N(\mathrm{end}(\bn),\mathrm{st}(\bm_1))\cdot\prod_{j=1}^sP_{l_j}^{(N)}(\bm_j).
\]
Furthermore, we set
\[
Z_N(\bk \mid)\coloneqq\zeta_{<(N+1)}^{}(\bk),\quad Z_N(\mid \bl)\coloneqq\zeta_{<(N+1)}^{\flat}(\bl).
\]
\begin{lemma}\label{lem:transport_connector}
For non-negative integers $n$ and $m$ satisfying $n\leq m\leq N$, we have
\begin{align}
\frac1n\cdot C_N(n,m)&=\sum_{n\leq b\leq m}C_N(n,b)\cdot\frac1b\quad(n>0),\label{eq:transport1}\\
\sum_{n<a\leq m}\frac1a\cdot C_N(a,m)&=\sum_{n\leq b<m}C_N(n,b)\cdot\frac1{N-b}\quad(n<m).\label{eq:transport2}
\end{align}
\end{lemma}
\begin{proof}
Since we can easily see that
\[
\frac1n\cdot\bigl(C_N(n,b)-C_N(n,b-1)\bigr)=C_N(n,b)\cdot\frac1b
\]
for $0<n<b\leq m$, and 
\[
\frac1a\cdot\bigl(C_N(a,b+1)-C_N(a,b)\bigr)=\bigl(C_N(a-1,b)-C_N(a,b)\bigr)\cdot\frac1{N-b}
\]
for $n<a\leq b< m$, we obtain the desired formulas.
\end{proof}
\begin{lemma}\label{lem:transport_connected_sum}
Let $k$ be a positive integer and $\bk$, $\bl$ indices.
Then we have
\begin{align*}
Z_N(\bk,k \mid \bl)&=Z_N(\bk\mid k,\bl), \\
Z_N(\bk,k\mid)&=Z_N(\bk\mid k), \\
Z_N(k\mid\bl)&=Z_N(\mid k,\bl).
\end{align*}
\end{lemma}
\begin{proof}
By applying \eqref{eq:transport1} $k-1$ times followed by a single application of \eqref{eq:transport2}, it follows from the definitions of connected sums.
\end{proof}
\begin{proof}[Proof of Theorem~$\ref{thm:main}$]
Write $\bk=(k_1,\dots,k_r)$.
By repeatedly using Lemma~\ref{lem:transport_connected_sum} $\dep(\bk)$ times, it is calculated as follows:
\begin{align*}
\zeta_{<(N+1)}^{}(\bk)&=Z_N(k_1,\dots,k_r\mid)\\
&=Z_N(k_1,\dots,k_{r-1}\mid k_r)\\
&=\cdots\\
&=Z_N(\mid k_1,\dots,k_r)=\zeta_{<(N+1)}^{\flat}(\bk).
\end{align*}
This completes the proof.
\end{proof}
\section{Duality in $\widehat{\cA}$}\label{sec:lifting}
As a lifting of the finite multiple zeta value $\zeta_{\cA}^{}(\bk)$, the \emph{$\bp$-adic finite multiple zeta value} $\zeta_{\widehat{\cA}}^{}(\bk)$ is defined as follows:
\[
\zeta_{\widehat{\cA}}^{}(\bk)\coloneqq ((\zeta_{<p}^{}(\bk)\bmod{p^n})_{p\in\PP})_{n\in\ZZ_{>0}}\in\widehat{\cA},
\]
where the \emph{$\bp$-adic number ring} $\widehat{\cA}$ is defined as
\[
\widehat{\cA}\coloneqq\varprojlim_{n}\left.\left(\prod_{p\in\PP} \ZZ/p^n\ZZ\right) \middle/ \left(\bigoplus_{p\in\PP} \ZZ/p^n\ZZ\right) \right. .
\]
This object was originally defined by Rosen~\cite{Rosen2015}, and up to the present, many liftings of relations among finite multiple zeta values to $\bp$-adic relations in $\widehat{\cA}$ have been obtained.
Note that there is a natural isomorphism $\widehat{\cA}/\bp\widehat{\cA}\simeq\cA$ and $\widehat{\cA}$ has the $\bp$-adic topology, where $\bp\in\widehat{\cA}$ is defined as $\bp\coloneqq((p\bmod{p^n})_{p\in\PP})_{n\in\ZZ_{>0}}$. 

Rosen proved a $\bp$-adic lifting of Theorem~\ref{thm:nonstarduality} and the second author proved a $\bp$-adic lifting of Theorem~\ref{thm:Aduality}:
\begin{theorem}[{Rosen~\cite[Theorem~4.5]{Rosen2015}}]
For an index $\bk$, we have
\[
\zeta_{\widehat{\cA}}^{}(\bk)+\sum_{i=0}^{\infty}(-1)^i\sum_{\bl}a_{\bl}\zeta_{\widehat{\cA}}^{}(\bl,1)\bp^{i+1}=(-1)^{\dep(\bk)}\sum_{\bk\preceq\bl}\zeta_{\widehat{\cA}}^{}(\bl),
\]
where $\bl$ runs over indices appearing in $\bk*(\{1\}^{i})=\sum a_{\bl}\bl$.
The product $*$ is the usual stuffle product as explained in \cite[\S2]{Kaneko2019}.
\end{theorem}
\begin{theorem}[{Seki~\cite[Theorem~1.3]{Seki2019}}]
For an index $\bk$, we have
\[
\sum_{i=0}^{\infty}\zeta^{\star}_{\widehat{\cA}}(\bk,\{1\}^i)\bp^i=-\sum_{i=0}^{\infty}\zeta^{\star}_{\widehat{\cA}}(\bk^{\vee},\{1\}^i)\bp^i,
\]
where $\zeta_{\widehat{\cA}}^{\star}(\bk)\coloneqq\sum_{\bl\preceq\bk}\zeta_{\widehat{\cA}}^{}(\bl)$.
\end{theorem}
Both proofs based on Hoffman's identity \eqref{eq:Hoffman_identity}.
Here, we present one $\bp$-adic lifting of Theorem~\ref{thm:nonstarduality} based on Theorem~\ref{thm:main}.
For an index $\bk=(k_1,\dots,k_r)$ and a tuple of non-negative integers $\bl=(l_1,\dots,l_r)$, we use the following notation:
\[
\bl\oplus\bk\coloneqq(l_1+k_1,\dots,l_r+k_r),\quad\bl\oslash\bk\coloneqq(l_1+1,\{1\}^{k_1-1},\dots,l_r+1,\{1\}^{k_r-1}),
\]
and $\wt(\bl)\coloneqq l_1+\cdots+l_r$.
\begin{theorem}\label{thm:p-adic_duality}
For an index $\bk$, we have
\[
\zeta_{\widehat{\cA}}^{}(\bk)=(-1)^{\dep(\bk)}\sum_{i=0}^{\infty}\sum_{\substack{\bl\in\ZZ_{\geq 0}^{\dep(\bk)} \\ \wt(\bl)=i}}\sum_{\bl\oplus\bk\preceq\bm\preceq\bl\oslash\bk}\zeta_{\widehat{\cA}}^{}(\bm)\bp^i.
\]
\end{theorem}
\begin{proof}
In the proof in Section~\ref{sec:Aderiving}, one may use a $p$-adic expansion
\[
\frac{1}{p-n}=-\sum_{l=0}^{\infty}\frac{p^l}{n^{l+1}}
\]
instead of \eqref{eq:easiest}.
\end{proof}
\subsection*{Acknowledgements}
The authors would like to express their gratitude to Professor Hidekazu Furusho and Professor Masanobu Kaneko for their valuable comments on the history of Theorem~\ref{thm:integral}.

\end{document}